\documentclass[12pt]{amsart}
\usepackage{amssymb}
\newtheorem{theorem}{Theorem}
\newtheorem{lemma}[theorem]{Lemma}

\newtheorem{corollary}[theorem]{Corollary}
\newtheorem{example}[theorem]{Example}
\newtheorem{question}{Question}[section]
\newtheorem{definition}[theorem]{Definition}

\title[Strongly discrete subsets]{Strongly discrete subsets with Lindel\"of closures}

\author[A. Bella]{Angelo Bella}

\address{ Dipartimento di Matematica e Informatica, viale A. 
Doria 6, 95125 Catania, Italy}
\email{bella@dmi.unict.it}

\author[S. Spadaro]{Santi Spadaro}

\address{ Dipartimento di Matematica e Informatica, viale A. 
Doria 6, 95125 Catania, Italy}
\email{santidspadaro@gmail.com}

\thanks{The authors were partially
supported by a grant from INdAM-GNSAGA}

\dedicatory{Dedicated to the memory of Phil Zenor}

\begin{document}

\maketitle

\begin{abstract} 

We define a topological space to be an \emph{SDL space} if the closure of each one of its strongly discrete subsets is Lindel\"of. After distinguishing this property from the Lindel\"of property we make various remarks about cardinal invariants of SDL spaces. For example we prove that $|X| \leq 2^{\chi(X)}$ for every SDL Urysohn space and that every SDL $P$-space of character $\leq \omega_1$ is regular and has cardinality $\leq 2^{\omega_1}$. Finally, we exploit our results to obtain some partial answers to questions about the cardinality of cellular-Lindel\"of spaces.

\end{abstract}

\section{Introduction}

One of the most elegant characterizations of compactness is the following: a space $X$ is compact if and only if the closure of every discrete subset of $X$ is compact. Whether this is true for the Lindel\"of property is the subject of a well-known question of Arhangel'skii, and the class of \emph{strongly discretely Lindel\"of spaces}, that is, spaces where closures of discrete sets are Lindel\"of, has received a certain amount of interest in the last few decades (see, for example \cite{A2}, \cite{AB}).

Inspired in part by this question, various authors have investigated variants of the Lindel\"of property which involve discrete sets or related objects. For example, the authors of \cite{JTW} define a space $X$ to be \emph{almost discretely Lindel\"of} if for every discrete set $D \subset X$ there is a Lindel\"of subspace $Y \subset X$ such that $D \subset Y$ and in \cite{BS} and \cite{B} we define a space $X$ to be (\emph{strongly}) \emph{cellular-Lindel\"of} if for every family $\mathcal{U}$ of pairwise disjoint non-empty open subsets of $X$ there is a (closed) Lindel\"of subspace $Y \subset X$ such that $U \cap Y \neq \emptyset$ for every $U \in \mathcal{U}$.

We introduce a new variant of the strongly discretely Lindel\"of property. Recall that a set $D \subset X$ is said to be \emph{strongly discrete} if for every $x \in D$ there is an open neighbourhood $U_x \subset X$ of $x$ such that $\{U_x: x \in D \}$ is a pairwise disjoint family.

\begin{definition}
We say that a space $X$ is SDL if the closure of every strongly discrete subset of $X$ is Lindel\"of.
\end{definition}

It is clear that every SDL space is strongly cellular-Lindel\"of and the converse is true if every point of the space has a disjoint local $\pi$-base. This happens for instance for first countable spaces or P-spaces of character $\omega_1$. 

In view of Arhangel'skii's celebrated theorem on the cardinality of Lindel\"of first-countable spaces  \cite{A} it is natural to ask whether there are restrictions on the cardinality of first-countable spaces satisfying any of the above weakenings of the Lindel\"of property. Every strongly discretely Lindel\"of Hausdorff first-countable space has cardinality at most continuum (a much stronger result than that was proved by the second author in \cite{S}). Moreover in \cite{BS} we proved that it is consistent that every almost discretely Lindel\"of first-countable Hausdorff space has cardinality at most continuum (Juh\'asz, Soukup and Szentmikl\'ossy \cite{JSS} proved that this is true in ZFC for regular spaces).  However, the following question is still open.

\begin{question} \label{mainquest}
Is the cardinality of a (regular) cellular-Lindel\"of first-countable space at most continuum?
\end{question}

Various partial answers to the above question are known. For example, define a space $X$ to be \emph{cellular-compact} if for every family $\mathcal{U}$ of pairwise disjoint non-empty open subsets of $X$ there is a compact subspace $K$ of $X$ such that $U \cap K \neq \emptyset$, for every $U \in \mathcal{U}$. Tkachuk and Wilson \cite{TW} proved that every cellular compact first-countable regular space has cardinality at most continuum. That was later generalized by Juh\'asz, Soukup and Szentmikl\'ossy \cite{JSS2}, who proved that if $X$ is a Hausdorff cellular compact space of countable closed pseudocharacter and countable tightness, then $X$ has cardinality at most continuum. Neither one of these results generalizes to higher cardinals, so it is natural to pose the following question:

\begin{question} \label{cellcom} {\ \\}
\begin{enumerate}
\item Is it true that $|X| \leq 2^{\psi_c(X) \cdot t(X)}$, for every cellular compact Hausdorff space $X$?
\item Is it at least true that $|X| \leq 2^{\chi(X)}$, for every cellular compact (regular) Hausdorff space $X$?
\end{enumerate}
\end{question}

Recall that $wL_c(X)$ is defined as the minimum cardinal $\kappa$ such that, for every closed set $F \subset X$ and for every open cover $\mathcal{U}$ of $F$ there is a $\leq \kappa$-sized subcollection $\mathcal{V}$ of $\mathcal{U}$ such that $F \subset \overline{\bigcup \mathcal{V}}$. We will prove that $wL_c(X) \leq t(X)$ for every SDL space $X$ and exploit that to prove that $|X| \leq 2^{\chi(X)}$ for every SDL Urysohn space and that every SDL $P$-space of character $\leq \omega_1$ is regular and has cardinality $\leq 2^{\omega_1}$. Moreover, we will offer a partial answer to Question $\ref{cellcom}$ by proving that $|X| \leq 2^{\psi_c(X) \cdot t(X)}$, for every cellular-compact space $X$ with a dense subset of isolated points. We finish by proving that SDL spaces with a $G_\delta$ diagonal of rank 2 have cardinality at most continuum.

For undefined notions see \cite{En} and \cite{J}. For background about elementary submodels see \cite{Dow}.

\section{The main results}

It is still an open question whether there is a strongly discretely Lindel\"of non-Lindel\"of space, but an SDL non-Lindel\"of space can be readily produced.

\begin{example}
There is a non-Lindel\"of SDL Tychonoff space.
\end{example}

\begin{proof} 
Let $X=\Sigma(2^\kappa)=\{x \in 2^\kappa: |x^{-1}(1)| \leq \aleph_0 \}$ with the topology induced from $2^\kappa$. Then $X$ is a countably compact non-compact space and hence it can't be Lindel\"of. Being dense in $2^\kappa$, $X$ is ccc and hence every strongly discrete subset of $X$ is countable. Fix a strongly discrete set $D \subset X$. Since $D$ is countable and each point of $D$ has countable support, there is a homeomorphic copy $K \subset 2^\kappa$ of $2^\omega$ such that $D \subset K$. Therefore $\overline{D} \subset K$, and hence the closure of every strongly discrete subset of $X$ is even compact.
\end{proof}

The following technical result will be needed in the applications of the SDL property to cellular-Lindel\"of spaces.

\begin{lemma} \label{SDLemma}
Let $X$ be a strongly cellular-Lindel\"of space with a disjoint local $\pi$-base at every point. Then $X$ is an SDL space.
\end{lemma}

\begin{proof}
Let $D \subset X$ be a strongly discrete set and let $\{U_x: x \in D \}$ be a pairwise disjoint open expansion of $D$. For every $x \in D$, let $\mathcal{U}_x$ be a disjoint local $\pi$-base at $x$. We may assume that every element of $\mathcal{U}_x$ is contained in $U_x$. Then $\mathcal{U}=\bigcup \{\mathcal{U}_x: x \in D\}$ is a cellular family of subsets of $X$ and hence there is a closed Lindel\"of subspace $L$ of $X$ such that $L \cap U \neq \emptyset$, for every $U \in \mathcal{U}$. But then every point of $D$ is an accumulation point of $L$ and therefore, since $L$ is closed, $D \subset L$, which implies that $\overline{D}$ is Lindel\"of.
\end{proof}

\begin{lemma} \label{FDLemma}
If $X$ is a Hausdorff first-countable space or a Hausdorff $P$-space of character $\leq \omega_1$ or a space with a dense set of isolated points then $X$ has a disjoint local $\pi$-base at every point.
\end{lemma}

\begin{proof}
Assume $X$ is a Hausdorff first-countable space. Fix a point $x \in X$. If $x$ is an isolated point there is nothing to prove, otherwise every open neighbourhood of $x$ contains at least another point. Since $X$ is first-countable, we can fix a countable decreasing local base $\{U_k: k < \omega \}$ at $x$. Let $U$ be any open neighbourhood of $X$. We claim that there is $n<\omega$ such that $U \nsubseteq \overline{U_n}$. Indeed, let $y \in U$ be a point distinct from $x$. Then there is $n<\omega$ and an open neighbourhood $V \subset U$ of $y$ such that $U_n \cap V=\emptyset$. But then $U \nsubseteq \overline{U_n}$, as, we wanted. Hence, for every $k<\omega$ we can find $n(k) < \omega$ such that $U_k \nsubseteq \overline{U_{n(k)}}$. Define $V_k=U_k \setminus \overline{U_{n(k)}}$, for every $k<\omega$. Then $\{V_k: k < \omega \}$ is a disjoint local $\pi$-base at $x$. The proof when $X$ is a regular $P$-space of character $\leq \omega_1$ is similar, since, in this case, $X$ has a decreasing local base of length $\omega_1$ at every point. Finally, if $X$ has a dense set $D$ of isolated points then $\{\{x\}: x \in D \}$ is even a disjoint $\pi$-base for the whole space.
\end{proof}

\begin{lemma} \label{mainlemma}
Let $X$ be an SDL space. Then $wL_c(X) \leq t(X)$.
\end{lemma}

\begin{proof} Let $F$ be a closed subset of $X$ and let $\mathcal{U}$ be  a collection of open subsets of $X$  such that $F \subseteq \bigcup \mathcal{U}$. 

Choose $x_0 \in F$ and take any $W_0=U_0 \in \mathcal{U}$ such that $x_0\in W_0$. Proceeding by induction, we define for each $\alpha <\kappa ^+$ points $x_\alpha
\in F$, open sets $W_\alpha \subseteq U_\alpha  \in \mathcal  U $
with
$x_\alpha \in W_\alpha $ and countable families  $\mathcal   V_\alpha\subseteq \mathcal  U$  in such a
way that the following conditions are satisfied.

\begin{enumerate}
\item  $\overline {\{x_\beta:\beta<\alpha \}} \subseteq \bigcup \mathcal  V_\alpha $;

\item $W_\alpha \cap (\bigcup(\{W_\beta:\beta<\alpha \}\cup  \bigcup \{\mathcal  V_\beta:\beta<\alpha \}))=\emptyset$
\end{enumerate}

Fix $\alpha < \kappa ^+$ and assume we have already defined 
$\{x_\beta:\beta<\alpha \}$, $\{W_\beta \subseteq U_\beta
:\beta<\alpha \}$   and
$\{\mathcal  V_\beta:\beta<\alpha \}$.  

If $F \subseteq 
\overline {\bigcup (\{W_\beta:\beta<\alpha \}\cup
\bigcup\{\mathcal 
V_\beta:\beta<\alpha \})}$ we stop because   $\mathcal 
V=\{U_\beta:\beta<\alpha \}\cup \bigcup\{\mathcal 
V_\beta:\beta<\alpha \}$ is a subfamily of $\mathcal 
U$ of cardinality not exceeding $\kappa $ 
satisfying $F\subseteq \overline {\bigcup\mathcal  V}$.

 If not,  
we may pick a point $x_\alpha \in  F$,  an open  set $W_\alpha$
and an element $U_\alpha  \in \mathcal  U
$ in such a way that $x_\alpha \in W_\alpha \subseteq U_\alpha $
and  
$W_\alpha \cap
(\bigcup(\{W_\beta:\beta<\alpha \}\cup \bigcup\{\mathcal 
V_\beta:\beta<\alpha \}))=\emptyset $. Finally, as the set     
$\{x_\beta:\beta<\alpha \}$ is strongly discrete, 
the set $\overline {\{x_\beta:\beta<\alpha \}}$ is Lindel\"of. 
Since $\overline {\{x_\beta:\beta<\alpha
\}}\subseteq F\subseteq \bigcup\mathcal  U$, there exists a
countable
family $\mathcal  V_\alpha
\subseteq \mathcal  U$ such that $\overline
{\{x_\beta:\beta<\alpha
\}}\subseteq \bigcup\mathcal  V_\alpha $.

Eventually the set $D=\{x_\alpha
:\alpha\in  \kappa ^+\}$ turns out to be a free sequence because
for
each $\alpha $ we have $\overline {\{x_\beta:\beta<\alpha
\}}\subseteq \bigcup \mathcal  V_\alpha $ and  $(\bigcup \mathcal
V_\alpha) \cap
\{x_\beta:\alpha \le \beta< \kappa ^+\}=\emptyset $.

The set $D$ is also strongly discrete and so it has a Lindel\"of closure. But a Lindel\"of space of tightness $\leq \kappa$ cannot have free sequences of cardinality $ \geq \kappa ^+$ and hence we reach a contradiction. Hence there is a subfamily $\mathcal  V\subseteq \mathcal  U$ of cardinality not exceeding $\kappa$ such that $F\subseteq \overline {\bigcup \mathcal{V}}$, as we wanted.

\end{proof}

Recall that a set $S \subset X$ is said to be \emph{$\theta$-dense} if for every non-empty open subset $O \subset X$ we have $\overline{O} \cap S \neq \emptyset$. We denote with $d_\theta(X)$ the smallest cardinality of a $\theta$-dense subset of $X$. The referee noted that the following lemma was proved by N. Carlson as Corollary 2.6 of \cite{C}, by a completely different argument. We include our proof of it anyway for the reader's convenience.
 
\begin{lemma} 
If $X$ is a  $T_2$ space, then $d_\theta(X) \le 2^{wL_c(X)\chi(X)}$.
\end{lemma}

\begin{proof}
Let $\kappa=wL_c(X) \cdot \chi(X)$. Let $\theta$ be a large enough regular cardinal and $M$ be a $\kappa$-closed elementary submodel of $H(\theta)$ such that $X \in M$, $\kappa+1 \subset M$ and $|M|=2^\kappa$.

\medskip

\noindent {\bf Claim 1.} $X \cap M$ is closed in $X$.

\begin{proof}[Proof of Claim 1]
Let $x \in \overline{X \cap M}$. Since $\chi(X) \leq \kappa$ we can fix a family $\{U_\alpha: \alpha < \kappa \}$ of open neighbourhoods of $x$ such that $\bigcap \overline{U_\alpha}=\{x\}$ and a set $Y \subset X \cap M$ having cardinality $\leq \kappa$ such that $x \in \overline{Y}$. Then $\{x\}=\bigcap_{\alpha < \kappa} \overline{Y \cap U_\alpha}$. Since $M$ is $\kappa$-closed, $Y \cap U_\alpha \in M$ and hence $\overline{Y \cap U_\alpha} \in M$,  for every $\alpha < \kappa$. Using $\kappa$-closure of $M$ again we infer that $x \in X \cap M$, which proves that $X \cap M$ is closed.
\end{proof}

\noindent {\bf Claim 2.} $X \cap M$ is $\theta$-dense in $X$.

\begin{proof}[Proof of Claim 2]
Suppose by contradiction this is not the case and let $O \subset X$ be a non-empty open set such that $\overline{O} \cap X \cap M=\emptyset$. For every $x \in X \cap M$ we can fix a local base $\mathcal{B}_x \in M$ of cardinality $\kappa$ for $x$. Since $\kappa+1 \subset M$ we have $\mathcal{B}_x \subset M$. Hence, for every $x \in X \cap M$ we can find an open neighbourhood $U_x$ of $x$ such that $U_x \in M$ and $U_x \cap O=\emptyset$. The family $\{U_x: x \cap M \}$ is an open cover of $X \cap M$ and hence, using $wL_c(X) \leq \kappa$ we can find a $\kappa$-sized set $C \subset X \cap M$ such that $X \cap M \subset \overline{\bigcup \{U_x: x \in C\}}$. Since $M$ is $\kappa$-closed, $C \in M$ and hence $M \models X \subset \overline{\bigcup \{U_x: x \in C\}}$. By elementarity it follows then that $H(\theta) \models X \subset \overline{\bigcup \{U_x: x \in C\}}$, which is a contradiction since $O \cap U_x=\emptyset$, for every $x \in C$.
\end{proof}
\end{proof}

The following lemma is proved by a standard argument, similar to the proof that $|X| \leq d(X)^{\chi(X)}$ for every Hausdorff space $X$.

\begin{lemma}   
If $X$ is a Urysohn space, then $|X| \leq d_\theta(X)^{\chi(X)}$. 
\end{lemma}

\begin{corollary} \label{maincor}
If $X$ is a Urysohn SDL space, then $|X| \leq 2^{\chi(X)}$. 
\end{corollary}

The referee noted that the above result also follows by combining Alas's result that $|X| \leq 2^{wL_c(X) \cdot \chi(X)}$ (see \cite{Al}, Theorem 1) with Lemma $\ref{mainlemma}$. 

The next result is a partial answer to the  fundamental open question on the cardinality of a cellular-Lindel\"of first-countable space.

\begin{corollary} 
\cite{B} If $X$ is a strongly cellular-Lindel\"of first countable Urysohn space, then $|X| \leq \frak c$.
\end{corollary}

\begin{proof}
It follows from Lemmas $\ref{SDLemma}$ and $\ref{FDLemma}$ that $X$ is an SDL space. Therefore the statement follows from Corollary $\ref{maincor}$.
\end{proof}

Two very natural questions remain open.

\begin{question}
Let $X$ be a Hausdorff SDL space. Is it true that $|X| \leq 2^{\chi(X)}$?
\end{question}

\begin{question}
Let $X$ be a strongly cellular-Lindel\"of first-countable Hausdorff space. Is it true that $|X| \leq \mathfrak{c}$?
\end{question}

\begin{lemma} \label{alaslemma2}
(\cite{Al}, Theorem 3) Let $X$ be a Hausdorff space with a dense set of isolated points. Then $|X| \leq 2^{\psi_c(X) \cdot t(X) \cdot wL_c(X)}$.
\end{lemma}

Combining the above Lemma with Lemma $\ref{mainlemma}$ we obtain a partial answer to Question $\ref{cellcom}$.

\begin{theorem}
Let $X$ be a strongly cellular-Lindel\"of Hausdorff space with a dense set of isolated points. Then $|X| \leq 2^{\psi_c(X) \cdot t(X)}$.
\end{theorem}

\section{$P$-spaces and $G_\delta$ diagonals}

Recall that a $P$-space is a space where $G_\delta$ sets are open. We prove that Corollary $\ref{maincor}$ does not require the Urysohn separation axiom for $P$-spaces of character $\leq \omega_1$.

\begin{theorem}
Every Hausdorff SDL P-space of character at most $\omega_1$ is regular.
\end{theorem}

\begin{proof}
Let $F$ be a closed set and $p\notin F$. Assume that $p$ is not isolated and fix a decreasing local base $\{U_\alpha :\alpha <\omega_1\}$ at $p$. 

It is enough to prove the existence of an ordinal $\gamma<\omega_1$ such that $\overline {U_\gamma}\cap F=\emptyset$. 

Suppose the contrary, set $\alpha _0=0$  and pick a point $x_0 \in \overline {U_{\alpha _0} }\cap F$.  Then, choose $\alpha _1$ such that $x_0\notin \overline {U_{\alpha _1}}$, put $V _0 =U_{\alpha_0}\setminus \overline {U_{\alpha _1}}$ and pick a point $x_1\in \overline {U_{\alpha _1}}\cap F$. Notice that $x_0\in \overline{V_0}$. Proceeding by induction, suppose that for some $\delta < \omega_1$ we have already chosen ordinals $\alpha _\xi$ and points $x_\xi$ for each $\xi<\delta$ in such a way that  $x_\xi\in \overline {V_\xi}\cap F$, where $V_\xi=U_{\alpha _\xi}\setminus \overline {U_{\alpha _\xi+1}}$.
To continue,  take $\alpha _\delta<\omega_1$ so that $x_\xi\notin \overline {U_{\alpha _\delta}}$ for each $\xi<\delta$  and pick $x_\delta\in \overline {U_{\alpha _\delta}}\cap F$, $\alpha_{\delta+1}>\alpha _\delta$ such that $x_\delta\notin \overline{U_{\alpha _\delta+1}}$ and set $V_\delta=U_{\alpha_\delta}\setminus \overline {U_{\alpha _\delta+1}}$.

At the end of this inductive process, for any $\delta<\omega_1$ fix a local base $\{U_\alpha ^\delta:\alpha <\omega_1\}$ at $x_\delta$. Take a non-empty open set $W_0^\delta\subseteq V_\delta \cap
U_0^\delta$ with $x_\delta \notin \overline {W^\delta_0}$. 
Assume we have already chosen non-empty pairwise disjoint open
sets $W^\delta_\beta\subseteq V_\delta\cap U_\beta^\delta$ with
$x_\delta\notin \overline {W_\beta^\delta}$ for each
$\beta<\alpha $. As $X$ is a P-space, the set $U^\delta_\alpha
\setminus \bigcup\{\overline {W_\beta^\delta}:\beta<\alpha \}$ is
open and non-empty. Therefore, there is a non-empty open set
$W_\alpha ^\delta\subseteq U_\alpha ^\delta\setminus
\bigcup\{\overline {W_ \beta ^\delta}:\beta<\alpha \}$ with
$x_\delta \notin \overline {W^\delta_\alpha }$.

The collection $\bigcup \{W^\delta_\alpha :\alpha ,\delta<\omega_1\}$ is a cellular family and hence we can obtain a strongly discrete set $D$ by selecting a point from each of its members. Since $X$ is an SDL space, the set $\overline D$ is
Lindel\"of. Moreover, since $\{U^\delta_\alpha :\alpha
<\omega_1\}$ is a local base at $x_\delta$ whose elements meet
$D$, we have $\{x_\delta:\delta<\omega_1\}\subseteq  \overline
D$. Thus, there exists a complete  accumulation point $z$ for the
set $\{x_\delta:\delta<\omega_1\}$. But then $z\in \bigcap
\{\overline {U_\alpha }:\alpha <\omega_1\}=\{p\}$ and hence
$z=p$. On the other hand,  from
$\{x_\delta:\delta<\omega_1\}\subseteq F$, it follows that $z\in F$, which is a contradiction.
\end{proof}

\begin{corollary}
If $X$ is a Hausdorff cellular-Lindel\"of P-space of character $\omega_1$, then $|X|\le 2^{\omega_1}$.
\end{corollary}

We finish with one more cardinal bound for SDL spaces which satisfy a mild generalized metric property, that of having a $G_\delta$ diagonal of rank 2. Given a family $\mathcal{U}$ of subsets of $X$ and a subset $F \subset X$ we define $St(F, \mathcal{U})=\bigcup \{U \in \mathcal{U}: F \cap U \neq \emptyset\}$. We also define $St(x, \mathcal{U})=St(\{x\}, \mathcal{U})$. Given a point $x \in X$ we define $St^n(x, \mathcal{U})$ by induction as follows: $St^1(x, \mathcal{U})=St(x, \mathcal{U})$ and $St^n(x, \mathcal{U})=St(St^{n-1}(x, \mathcal{U}), \mathcal{U})$ for every $n>1$. A space $X$ is said to have a $G_\delta$ diagonal of rank n if there is a sequence $\{\mathcal{U}_k: k < \omega \}$ of open covers of $X$ such that, for every $x \in X$, $\bigcap \{St^n(x, \mathcal{U}_k): k < \omega \}=\{x\}$.

Recall the following fundamental result in infinite combinatorics, a proof of which can be found, for example in \cite{Je}, Theorem 9.6.

\begin{theorem}
(The Erd\H{o}s-Rado Theorem) Let $X$ be a set of cardinality larger than continuum and $f: [X]^2 \to \omega$ be a function. Then there is an uncountable subset $S \subset X$ and an integer $m<\omega$ such that $f([S]^2)=\{m\}$.
\end{theorem}

\begin{theorem} \label{rank2}
Every SDL space with a $G_\delta$ diagonal of rank $2$ has cardinality at most $\mathfrak{c}$.
\end{theorem}

\begin{proof}
Let $\{\mathcal{U}_k: k < \omega \}$ be a sequence of open covers witnessing a $G_\delta$ diagonal of rank 2. Let $x, y \in X$ be distinct points. Then there is $n<\omega$ such that $x \notin St^2(y, \mathcal{U}_n)=St(St(y, \mathcal{U}_n), \mathcal{U}_n)$. Therefore, for every $U \in \mathcal{U}_n$ such that $x \in U$ we have $U \cap St(y, \mathcal{U}_n)=\emptyset$. But $St(x, \mathcal{U}_n)=\bigcup \{U \in \mathcal{U}_n: x \in U\}$ and hence $St(x, \mathcal{U}_n) \cap St(y, \mathcal{U}_n)=\emptyset$.

Suppose by contradiction that $|X| > \mathfrak{c}$. For every pair of distinct points $x, y \in X$ let $f(\{x,y\})=\min\{n<\omega: St(x, \mathcal{U}_n) \cap St(y, \mathcal{U}_n)=\emptyset\}$. Then $f$ is a map from $[X]^2$ to $\omega$ and hence, by the Erd\H{o}s-Rado theorem there is an uncountable set $S \subset X$ and a number $m< \omega$ such that $f([S]^2)=\{m\}$. But then $\{St(x, \mathcal{U}_m): x \in S \}$ is a cellular family and hence $S$ is a strongly discrete set. Moreover, $S$ is closed. Indeed, let $p \in X \setminus S$. Then there is $U \in \mathcal{U}_m$ such that $p \in U$. Note that by definition of $S$ we have $|U \cap S| \leq 1$, and that shows that $p$ is not an accumulation point for $S$. But an SDL space cannot have an uncountable strongly discrete closed subset and we are done.
\end{proof}

Note that in the above theorem the SDL property can be replaced with the property that every strongly discrete closed subset of $X$ is countable. Given that every SDL space is strongly cellular-Lindel\"of it is natural to ask the following question.

\begin{question}
Is it true that every strongly cellular-Lindel\"of space with a $G_\delta$ diagonal of rank 2 has cardinality at most $\mathfrak{c}$?
\end{question}

We recall that the following question is also open. In \cite{BS2} we proved that it is consistent that every normal cellular-Lindel\"of space with a $G_\delta$ diagonal of rank 2 has cardinality at most continuum and a few more partial answers were given by Xuan and Song in \cite{XS2}.

\begin{question}
Is it consistent that every cellular-Lindel\"of space with a $G_\delta$ diagonal of rank 2 has cardinality at most $\mathfrak{c}$?
\end{question}

A classical result due to Ginsburg and Grant Woods \cite{GW} states that a space with a $G_\delta$ diagonal containing no uncountable closed discrete subsets has cardinality at most continuum. Discrete cannot be replaced with strongly discrete in Ginsburg and Grant Woods's result since there are ccc spaces with a $G_\delta$ diagonal and arbitrarily large cardinality (see \cite{Sh} and \cite{U}). However, the following question seems to be open.

\begin{question}
Is it true that every SDL space with a $G_\delta$ diagonal has cardinality at most continuum?
\end{question}

\section{Acknowledgements}

The authors are grateful to INdAM-GNSAGA for financial support and to the anonymous referee for various helpful suggestions.

\end{document}